\documentclass{amsart}

\usepackage{amsmath}
\usepackage{amsthm}
\usepackage{amsfonts}
\usepackage{amssymb}
\usepackage{graphicx}
\usepackage{caption2}

\newcommand\R{{\mathbb{R}}}

\renewcommand\P{{\mathbf{P}}}
\newcommand\E{{\mathbf{E}}}

\newcommand\eps{{\varepsilon}}

\newcommand\tr{\operatorname{trace}}

\theoremstyle{plain}
  \newtheorem{theorem}{Theorem}[section]
  \newtheorem{conjecture}[theorem]{Conjecture}

  \newtheorem{proposition}[theorem]{Proposition}
  
  \newtheorem{lemma}[theorem]{Lemma}

\theoremstyle{remark}
  \newtheorem{remark}[theorem]{Remark}

\theoremstyle{definition}
  \newtheorem{definition}[theorem]{Definition}

\include{psfig}

\begin{document}

\title[Necessity of four moments]{Random matrices: Localization of the eigenvalues and the necessity of four moments}

\author{Terence Tao}
\address{Department of Mathematics, UCLA, Los Angeles CA 90095-1555}
\email{tao@math.ucla.edu}
\thanks{T. Tao is is supported by a grant from the MacArthur Foundation, by NSF grant DMS-0649473, and by the NSF Waterman award.}

\author{Van Vu}
\address{Department of Mathematics, Rutgers, Piscataway, NJ 08854}
\email{vanvu@math.rutgers.edu}
\thanks{V. Vu is supported by research grants DMS-0901216 and AFOSAR-FA-9550-09-1-0167.}

\begin{abstract}  Consider the eigenvalues $\lambda_i(M_n)$  (in increasing order) of a random Hermitian matrix $M_n$ whose upper-triangular entries are independent with mean zero and variance one, and are exponentially decaying.  By Wigner's semicircular law, one expects that $\lambda_i(M_n)$ concentrates around $\gamma_i \sqrt n$, where $\int_{-\infty}^{\gamma_i} \rho_{sc} (x) dx = \frac{i}{n}$ and $\rho_{sc}$ is the semicircular function. 

In this paper, we show that if the entries have vanishing third moment, then   for all $1\le i \le n$

$$\E |\lambda_i(M_n)-\sqrt{n} \gamma_i|^2  = O( \min( n^{-c} \min(i,n+1-i)^{-2/3} n^{2/3}, n^{1/3+\eps}  ) ) ,$$  for some  absolute constant 
$c>0$ and any  absolute constant $\eps>0$. In particular, for the eigenvalues in the bulk ($\min \{i, n-i\}=\Theta (n)$), 
$$\E |\lambda_i(M_n)-\sqrt{n} \gamma_i|^2  = O( n^{-c}). $$

\noindent  A similar result is achieved for the rate of convergence. 

As a corollary,  we show that  the four moment condition in the Four Moment Theorem 
is necessary, in the sense that if one allows the fourth moment to change (while keeping the first three moments fixed), then the \emph{mean} of $\lambda_i(M_n)$ changes by an amount comparable to $n^{-1/2}$ on the average.  We make a precise conjecture about how the expectation of the eigenvalues vary with the fourth moment. \end{abstract}

\maketitle

\setcounter{tocdepth}{2}

\noindent {\it Key words: Random matrices, localization, rate of convergence, four moment theorem .} 

\section{Introduction}

This note is concerned with the local eigenvalue statistics of the following random matrix model.

\begin{definition}[Wigner matrices]\label{herm-def}  A \emph{Wigner matrix} is a random hermitian matrix $M_n = (\zeta_{ij})_{1 \leq i, j \leq n}$ such that
\begin{itemize}
\item The $\zeta_{ij}$ for $1 \leq i \leq j \leq n$ are independent with mean zero and variance one, and $\zeta_{ji} = \overline{\zeta_{ij}}$;
\item For $1 \leq i < j \leq n$, $\zeta_{ij}$ are identically distributed, with the real and imaginary parts of $\zeta_{ij}$ being independent and identically distributed with distribution $\eta$;
\item For $1 \leq i \leq n$, the $\zeta_{ii}$ are identically distributed with distribution $\tilde \eta$;
\item  (Uniform exponential decay) There exist constants $C, C' > 0$ such that
\begin{equation}\label{ued}
\P( |\zeta_{ij}| \ge t^C) \le \exp(- t)
\end{equation}
for all $t \ge C'$ and $1 \leq i,j \leq n$.
\end{itemize}
We refer to $\eta,\tilde \eta$ as the \emph{atom distributions} of $M_n$.
\end{definition}

A classical example of a Wigner matrix is the \emph{Gaussian unitary ensemble} (GUE), in which $\eta$ and $\tilde \eta$ being the normal distribution with mean zero and variances $1/2$, $1$ respectively. 

 \begin{remark} Wigner's matrices are not the most general random matrix model for which the results here are applicable, but we restrict to this case for simplicity. The results, for example, hold for real matrices, in particular Bernoulli matrices. \end{remark} 

A  Wigner matrix $M_n$ has $n$ real eigenvalues
$$ \lambda_1(M_n) \leq \ldots \leq \lambda_n(M_n).$$

The {\it global} distribution of these eigenvalues has been known since the 1950s, and is  described by 
  the famous \emph{Wigner semicircular law}, which asserts that the empirical spectral measure
$$ \frac{1}{n} \sum_{i=1}^n \delta_{\frac{1}{\sqrt n} \lambda_i(M_n) }$$
converges almost surely (in the vague topology) to the semicircular distribution $\rho_{sc}(x)\ dx$, where
$$ \rho_{sc}(x) := \frac{1}{\pi} (4-x^2)_{+}^{1/2}.$$

It is of interest to understand the distribution of individual eigenvalues $\lambda_i(M_n)$.  If for each $1 \leq i \leq n$ we define the \emph{classical location} $\gamma_i$ of the normalised $i^{th}$ eigenvalue by the formula
\begin{equation}\label{gammai}
\int_{-\infty}^{\gamma_i} \rho_{sc} (x) dx = \frac{i}{n}. 
\end{equation}
then the Wigner semicircular law (combined with an almost sure bound of $(2+o(1))\sqrt{n}$ for the operator norm of $M_n$, due to Bai and Yin\cite{baiyin}) is equivalent to the assertion that one has
\begin{equation}\label{gammai-form}
\lambda_i(M_n) = \gamma_i \sqrt{n} + o(\sqrt{n})
\end{equation}
uniformly for $1 \leq i \leq n$, almost surely as $n \to \infty$.  If we ignore the $o(\sqrt{n})$ error in \eqref{gammai-form}, we are thus led to the heuristic
\begin{equation}\label{approx}
\lambda_{i+1}(M_n)-\lambda_i(M_n) \approx \min(i,n-i)^{-1/3} n^{-1/6}
\end{equation}
for the $i^{th}$ eigenvalue spacing.  In particular, this spacing should be comparable to $n^{-1/2}$ in the bulk region $\delta n \leq i \leq (1-\delta) n$ (for any fixed $\delta>0$), and as large as $n^{-1/6}$ at the edge of the spectrum.  Note though that this derivation is non-rigorous as the $o(\sqrt{n})$ error could be much larger than the expected gap size \eqref{approx}.

In the last few years, there has been much progress in formalising the above heuristics and obtaining more precise control on the distribution of the eigenvalues and their spacings: see e.g. \cite{Sos1}, \cite{ERSY2}, \cite{TVbulk}, \cite{ERSTVY}, \cite{TVedge}, \cite{Joh2}.  A recent survey of these topics can be found in \cite{guionnet}.

It is natural to ask whether the $o(\sqrt{n})$ error in \eqref{gammai-form} can be improved.  In \cite{Vu}, the Talagrand concentration inequality was used to establish (among other things\footnote{Strictly speaking, the results in \cite{Vu}, \cite{Meckes} only establish the bounds \eqref{vu}, \eqref{mox} implicitly, and require in addition that the matrix entries are uniformly bounded, rather than exponentially decaying.  However, the arguments in these papers can be easily extended to the exponentially decaying case, after a standard truncation argument to reduce to the case when the entries are bounded in magnitude by $n^{\eps/2}$ (say), and replacing the median with a slightly shifted variant.  We omit the details.}) that
\begin{equation}\label{vu}
\lambda_i(M_n) = {\bf M} \lambda_i(M_n) + O( n^\eps \min( i, n+1-i ) )
\end{equation}
with probability at least $1-O_{A,\eps}(n^{-A})$ (say) for any $A, \eps > 0$, where ${\bf M} \lambda_i(M_n)$ is the median of $\lambda_i(M_n)$; see also \cite{GZ} for closely related results.  This was improved in \cite{Meckes} to
\begin{equation}\label{mox}
\lambda_i(M_n) = {\bf M} \lambda_i(M_n) + O( n^\eps \min( i, n+1-i )^{1/2} )
\end{equation}
with the same probability of $1-O_{A,\eps}(n^{-A})$.  In the bulk region $\delta n < i < (1-\delta)n$, the concentration of measure arguments in \cite{GZ} give the bound 
\begin{equation}\label{lamn}
 \lambda_i(M_n) = \sqrt{n} \gamma_i + O( n^{1/2+\eps} \min(i,n+1-i)^{-1/3} n^{-1/6} ) 
\end{equation}
with probability $1-O_{A,\eps,\delta}(n^{-A})$ whenever $\min(i,n+1-i) > n^{1/2+\eps}$ (see Section \ref{local-proof} for further discussion of this bound).

In all the above estimates, the error term is larger than $1$.  In the recent paper \cite[Theorem 7.1]{EYY}, the bound
\begin{equation}\label{logsob}
\sum_{i=1}^n \E |\lambda_i(M_n) - \gamma_i \sqrt{n}|^2 = O( n^{1-c} )
\end{equation}
was established for some absolute constant $c>0$. (See also the earlier result in \cite[Theorem 6.3]{EYY0}, which established \eqref{logsob} under an additional log-Sobolev hypothesis on the distribution.)  In the bulk region $\delta n < i < (1-\delta) n$, a significantly stronger localisation was obtained in \cite{EYY} (see the equation preceding $(7.8)$ in that paper), namely that $\lambda_i(M_n) = \gamma_i \sqrt{n} + O( n^{-1/2+\eps} )$ with probability $O_{\eps,A,\delta}(n^{-A})$ for any $\eps, A$, with variants of this result also being obtained closer to the edge.  This result was established via a strong bound on the convergence of the Stieltjes transform, which in turn was obtained by a lengthy moment method computation.

Our first main result gives an alternate method to establish eigenvalue localisation, based on the \emph{three moment theorem} rather than on combining the Stieltjes transform method with the moment method, and which gives a non-averaged version \eqref{logsob} (in the case when the third moment vanishes):

\begin{theorem}[Localisation]\label{stronglocal}  
There is an absolute constant $c>0$ such that the following holds for any constant $\eps >0$. 
Let  $M_n$ be  a Wigner matrix whose atom distribution $\eta$ has vanishing third moment $\E \eta^3 = 0$.
Then for all $1\le i \le n$, 
\begin{equation}\label{lamin}
 \E |\lambda_i(M_n)-\sqrt{n} \gamma_i|^2  = O( \min( n^{-c} \min(i,n+1-i)^{-2/3} n^{2/3}, n^{1/3+\eps}  ) ).
\end{equation}
\end{theorem} 

One can set $c$ to be $1/1000$ (say) and we make no attempt to optimize this constant. 
 By summing over $i$, one obtains  \eqref{logsob}. 
 Furthermore, in the bulk region  $\delta n \leq i \leq (1-\delta) n$, the theorem implies 
$$\E |\lambda_i(M_n)-\sqrt{n} \gamma_i|^2  = O_\delta(n^{-c}). $$
This is not as strong as the recent localisation result obtained in \cite{EYY}, but the proof is shorter (assuming the three moment theorem) and will suffice for our applications.  In view of \eqref{approx}, the optimal bound on the right-hand side of \eqref{lamin} should be $O( \min(i,n-i)^{-2/3} n^{-1/3+\eps} )$.

Let $N_I$ be the number of eigenvalues (of $\frac{1}{\sqrt n} M_n$) in $I$, and define $F_n(x) = \frac{1}{n} \E N_{[-2,x]} $. 
The quantity $\Delta:= \sup_x  |F_n(x) -\int_{-\infty}^ x \rho _{sc}(t) dt | $ is of interest and has been investigated
by many researchers (see  \cite[Chapter 8]{BS}, \cite{GT} and the references therein). In these papers, it has been shown that 
$\Delta  = O( n^{-1/2})$ under various assumptions (the most general one seems to be in \cite{GT} which only requires bounded fourth moment). 

The arguments in the proof of Theorem \ref{stronglocal} can be used to break the $n^{-1/2}$ barrier, under the extra third moment condition:

\begin{theorem} \label{rateofconvergence} 
There is an absolute constant $c>0$ such that the following holds. 
Let  $M_n$ be  a Wigner matrix whose atom distribution $\eta$ has vanishing third moment $\E \eta^3 = 0$.
Then $\Delta = O(n^{-1/2 -c} )$. 
\end{theorem} 

We prove this theorem in Section \ref{rate-sec}. The bound $n^{-1/2 -c}$ can be improved to $n^{-1 +\eps}$ if we use median instead of expectation in the definition of $F_n$ (see Remark \ref{remark:median}).   A related result was proven (using different methods) recently in \cite[Theorem 6.3]{EYY}, namely that
$$ \frac{1}{n} N_{[-2,x]} = \int_{-\infty}^x \rho_{sc}(t)\ dt + O( n^{-1/2-c} / |x-2| )$$
with probability $1-O(n^{-A})$ for any fixed $A$, without a third moment hypothesis.

Next, we give an application of Theorem \ref{stronglocal} to demonstrate the sharpness (in some sense) of the \emph{four moment theorem}, introduced by the authors in \cite{TVbulk, TVedge} in order to study the distribution
 of eigenvalues of random matrices.  We state a special case of this theorem here:
 
\begin{theorem}[Four Moment Theorem]\label{theorem:main2} For all sufficiently small $c_0 > 0$ the following holds.   Let $M_n$, $M'_n$ be two Wigner random matrices whose atom distributions $\eta, \eta'$ have matching moments to fourth order, thus $\E \eta^j = \E (\eta')^j$ for $j=3,4$.  Let $G: \R \to \R$ be a smooth function obeying the derivative bounds
\begin{equation}\label{G-deriv}
|G^{(j)}(x)| \leq n^{c_0}
\end{equation}
for all $0 \leq j \leq 5$ and $x \in \R$.   We abbreviate $\lambda_i := \lambda_i(M_n)$ and $\lambda'_i := \lambda_i(M'_n)$.

Then for $n$ sufficiently large (depending on $c_0$ and the constants $C,C'$ in \eqref{ued}) and all $1 \leq i \leq n$ one has
\begin{equation}\label{equation:approximation}
|\E  G(\sqrt{n} \lambda_{i} )
-  \E  G(\sqrt{n} \lambda'_{i})| \le n^{-c_0}.
\end{equation}
If the atom distributions $\eta, \eta'$ have matching moments only to order $3$ rather than $4$ (i.e. $\E \eta^3 = \E (\eta')^3$), then \eqref{equation:approximation} still holds provided that one strengthens \eqref{G-deriv} to
\begin{equation}\label{g2}
|G^{(j)}(x)| \leq n^{-C j c_0}
\end{equation}
for all $0 \leq j \leq 5$ and $x \in \R$, and some absolute constant $C$.
\end{theorem}

\begin{proof} See \cite[Theorem 15]{TVbulk}  (which handled the bulk case when $\delta n < i < (1-\delta) n$) and \cite[Theorem 1.13]{TVedge} (which handled the edge case).  These theorems can also handle the joint distribution of several eigenvalues at once, as well as somewhat more general ensembles than those in Definition \ref{herm-def}, but we will not discuss these generalisations here.  
\end{proof}

We will refer to the second part of Theorem \ref{theorem:main2} as the \emph{three moment theorem}.

Roughly speaking, Theorem \ref{theorem:main2} asserts that the distributions of $\lambda_i(M_n)$ and $\lambda_i(M'_n)$ differ by $O(n^{-1/2-c})$ for some $c>0$ if the atom distributions have matching moments to order $4$, and by $O_c(n^{-1/2+c})$ for any $c>0$ if the atom distributions only have matching moments to order $3$.  For instance, for sufficiently large $n$ one has
$$  \P( \lambda_i \leq a ) \leq \P( \lambda_i' \leq a + n^{-1/2-c} ) + n^{-c} $$
for some $c>0$ if one has matching moments to order $4$, and
$$  \P( \lambda_i \leq a ) \leq \P( \lambda_i' \leq a + n^{-1/2+c} ) + n^{-c} $$
for all $c>0$ (with $n$ sufficiently large depending on $c$) if one has matching moments to order $3$.  Morally speaking, this means that the medians ${\bf M} \lambda_i, {\bf M} \lambda'_i$ of $\lambda_i$, $\lambda_i'$ differ by $O(n^{-1/2-c})$ when there are four matching moments and by $O_c(n^{-1/2+c})$ when there are three matching moments, although this is not quite rigorous due to the presence of the $n^{-c}$ error in the above bounds.  (For some rigorous bounds on the median of $\lambda_i$, see Section \ref{local-proof}.)

The matching moment conditions are essential to the method of proof of Theorem \ref{theorem:main2}, which uses a Taylor expansion argument.  But it is natural to ask if these conditions are in fact necessary.  Indeed, if one is not interested in the distribution of \emph{individual} eigenvalues $\lambda_i(M_n)$, but instead in the
 \emph{$k$-point correlation functions} of these eigenvalues, then in the asymptotic limit $n \to \infty$ (and with appropriate normalisations), these correlation functions have a universal distribution regardless of how many matching moments there are; see \cite{TVbulk, ERSTVY} (with earlier partial results in this direction in \cite{Joh}, \cite{ERSY2}). It is not hard to see that the universality for the 
 limiting distributions (or joint distributions) of individual eigenvalues imply the 
 universality of the $k$-point correlation function. 

Our second main result below is evidence that the four moment hypothesis is indeed necessary if one wishes to control individual eigenvalues at the scale of the eigenvalue spacing \eqref{approx}.

\begin{theorem}[Necessity of fourth moment hypothesis]\label{theorem:main}  Let $M_n, M'_n$ be Wigner matrices whose atom variables  $\eta, \eta'$
satisfy  $\E \eta^3 = \E (\eta')^3=0$ but their fourth moment are different 
$\E \eta^4 \neq \E (\eta')^4$.  As before, write $\lambda_i := \lambda_i(M_n)$ and $\lambda'_i := \lambda_i(M'_n)$.  Then for all sufficiently large $n$, one has
$$ \sum_{i=1}^n | \E \lambda_i - \E \lambda_i'  | \geq \kappa n^{1/2}$$
for some $\kappa$ depending only on the atom distributions.
In particular (by the pigeonhole principle), there exists $1 \leq i \leq n$ such that
$$ | \E \lambda_i - \E \lambda_i'  | \geq \kappa' \min(i,n+1-i)^{-1/3} n^{-1/6},$$
where $\kappa' > 0$ depends only on the atom distributions.
\end{theorem}

Theorem \ref{theorem:main} is not exactly comparable to the four moment theorem, as it pertains to the \emph{mean} of the eigenvalues $\lambda_i$, whereas the four moment theorem instead controls quantities such as the \emph{median}.  However, it is expected that the mean and median of $\lambda_i$ should be quite close (in particular, closer than the expected eigenvalue spacing \eqref{approx}), but the best known concentration results for $\lambda_i$ (such as Theorem \ref{stronglocal}) are not strong enough to establish this yet.  If one assumes that the mean and median are sufficiently close, then Theorem \ref{theorem:main} is strong evidence that the four moment theorem breaks down if one only assumes three matching moments.

The three moment theorem implies (roughly speaking) that the medians of $\lambda_i, \lambda'_i$ should only differ by $O(n^{-1/2+c_0})$, for arbitrarily small $c_0>0$.  In view of this, one expects the index $i$ provided by Theorem \ref{theorem:main} to lie in the bulk region $\delta n < i < (1-\delta) n$, and indeed the conclusion Theorem \ref{theorem:main} should in fact hold for \emph{most} $i$ in this bulk region.  However, we were unable to demonstrate this.  Nevertheless, concentration bounds such as those given earlier in this section should be able to establish some non-trivial lower bound on $\min(i,n+1-i)$.

The question that  how each particular eigenvalue reacts to a change in  the fourth moment looks very interesting. By utilising higher moments and making some heuristic arguments (see the last section of the paper)  we are led to the following precise conjecture.

\begin{conjecture}[Conjectured asymptotic]\label{conj-asym}  For $\delta n \leq i \leq (1-\delta) n$ with $\delta>0$ fixed, one has
$$ \E \lambda_i = n^{1/2} \gamma_i + n^{-1/2} C_{i,n} + \frac{1}{4 \sqrt{n}} (\gamma_i^3 - 2 \gamma_i) \E \eta^4 + O_\delta(n^{-1/2-c} )$$
for some absolute constant $c>0$, where $C_{i,n}$ is some bounded quantity depending only on $i, n$ (and is in particular independent of $\eta$).  The same statement should also be true for the median ${\bf M} \lambda_i$.
\end{conjecture}
 We ran a numerical test to check the conjecture. We generated two random matrices models with whose entry's distributions are Gaussian $N(0,1)$ and Laplace $(0,1/\sqrt{2})$ and investigated the behaviour of the difference $\E \lambda_i-\E \lambda_i'$ as a function of $i$.
 
\begin{figure}[htbp]
  \centering 
  \includegraphics[scale=0.9]{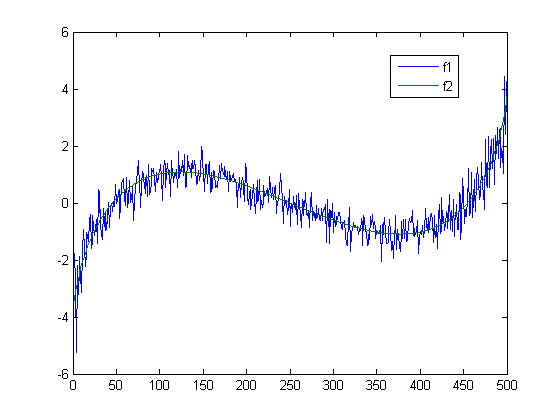}
  \captionstyle{center} 
  \onelinecaptionsfalse 
  \caption{Comparison between 2 curves $f_1=4\sqrt{n}\frac{\E \lambda_i-\E \lambda_i'}{\E\eta^4-\E\eta'^4}$ and $f_2=\gamma_i^3 - 2 \gamma_i$. ($n=500$)} 
\label{fig:SCL1}
 \end{figure}
 
This conjecture would imply that if one increases the fourth moment, then (in expectation) those $\lambda_i$ with $\gamma_i  \le -\sqrt 2$ or $0 \le \gamma_i \le \sqrt 2$ are shift to the left (decreasing), while those $\lambda_i$ with $\gamma_i  \ge \sqrt 2$ or $0 \ge  \gamma_i \ge -\sqrt 2$ are shifted to the right (increasing). In other words, the eigenvalues in the middle move toward the center of the spectrum, while those closer to the edge move outward. 

We prove Theorem \ref{theorem:main} in Section \ref{mainsec}.  Apart from Theorem \ref{stronglocal}, the main ingredient is a standard moment computation (Lemma \ref{fourth}) that compares $\sum_{i=1}^n \lambda_i^4$ with $\sum_{i=1}^n (\lambda'_i)^4$.


{\it Acknowledgement.} We would like to thank K. Johansson for a useful conversation and O. Zeitouni for confirming Lemma \ref{slick}.  Thanks to R. Killip, M. Meckes, and H.T. Yau for corrections.

{\it Notation.}  We use the usual asymptotic notation as $n \to \infty$, thus $O(f(n))$ denotes a quantity $g(n)$ bounded in magnitude by $C f(n)$, and $o(f(n))$ denotes a quantity $g(n)$ bounded in magnitude by $c(n) f(n)$, where $c(n) \to 0$ as $n \to \infty$.  If we need, $c, C$ to depend on additional parameters, we indicate this by subscripts, e.g. $o_k(f(n))$ is a quantity bounded in magnitude by $c_k(n) f(n)$, where $c_k(n) \to 0$ as $n \to \infty$ for each fixed $k$.

\section{Proofs of Theorems \ref{stronglocal} and \ref{rateofconvergence}}\label{local-proof}

We now prove Theorem \ref{stronglocal}.  We may assume without loss of generality that $\eps>0$ is small; all implied constants can depend on $\eps$, and we assume that $n$ is sufficiently large depending on $\eps$.

We first observe that the claim is true in the edge cases $i=1,n$.  Indeed, in those cases the estimate \eqref{mox} (or \eqref{vu}) gives
$$
\lambda_i(M_n) = {\bf M} \lambda_i(M_n) + O( n^{\eps} )$$
with probability $1-O(n^{-100})$.  On the other hand, the Tracy-Widom law for Wigner matrices (see \cite{TVedge}) gives
$$
{\bf M} \lambda_i(M_n) = \gamma_i \sqrt{n} + O(n^{-1/6})$$
at the edge, and the claim follows (with plenty of room to spare).

We may now reduce to the bulk case $\min(i,n-i) \geq n^{1/2+\eps}$.  Indeed, one can deduce the edge case $1 < i < n^{1/2+\eps}$ from the bulk case by setting $i_0$ to be the least integer greater than $n^{1/2+\eps}$, and using the crude pointwise bound
$$
|\lambda_i(M_n)-\sqrt{n} \gamma_i|
\leq |\lambda_1(M_n)-\sqrt{n} \gamma_1| + |\lambda_{i_0}(M_n)-\sqrt{n} \gamma_{i_0}| + |\sqrt{n} \gamma_1 - \sqrt{n} \gamma_{i_0}|$$
and observing that $\sqrt{n} \gamma_1 - \sqrt{n} \gamma_{i_0} = O( n^{1/6 + \eps/3} )$.  Similarly to deal with the case $n - n^{(1+\eps)/2} < i < n$.

Henceforth we fix $i$ with $\min(i,n-i) \geq n^{1/2+\eps}$.  The next step is to verify the theorem in the model case that $M_n$ is the GUE random matrix ensemble.  In this case, much sharper concentration results are known.  Indeed, we have

\begin{lemma}[Concentration for GUE]\label{slick} Let $M_n$ be a GUE matrix, and let $I \subset \R$ be an interval.  Let $N_I$ be the counting function $N_I := \{ 1 \leq i \leq n: \frac{1}{\sqrt{n}} \lambda_i(M_n) \in I \}$.
Then one has
$$ \P( | N_I - n \int_I \rho_{sc}(x)\ dx| \geq n^\eps ) \leq n^{-100} $$
(say) uniformly in $I$, if $n$ is sufficiently large depending on $\eps$.
\end{lemma}

\begin{proof}
This follows from the fact that the number of eigenvalues of GUE in $I$ can be expressed\footnote{In fact, in the GUE case $N_I$ has a binomial distribution, though we will not need this fact here.} as the sum of independent random variables (see \cite[Corollary 4.2.24]{AGZbook}), with variance of logarithmic size (see \cite[Lemma 2.3]{Gus}) and thus strongly concentrated. For details, see \cite{AGZbook}.  In the bulk region $I \subset [-2+\delta,2-\delta]$, this type of result (for more general Wigner ensembles) was established in \cite[Theorem 6.3]{EYY}.
\end{proof}

From this lemma and a standard computation, we conclude that in the GUE case one has
\begin{equation}\label{limi}
 \lambda_i(M_n) = \sqrt{n} \gamma_i + O( n^\eps \min(i,n+1-i)^{-1/3} n^{-1/6} ) 
\end{equation}
with probability $1-O(n^{-100})$.  
From this bound (and using very crude estimates to control the tail event of probability $O(n^{-100})$, e.g. controlling $\lambda_i(M_n)$ by the Frobenius norm of $M_n$), we have
$$ \E |\lambda_i(M_n)-\sqrt{n} \gamma_i|^2  = O( n^\eps \min(i,n+1-i)^{-1/3} n^{-1/6} )^2 + O(n^{-10})$$
(say), which easily implies \eqref{lamin} (with some room to spare).

\begin{remark}  In the bulk case $\delta n < i < (1-\delta) n$, a result of Gustavsson \cite{Gus} shows the related statement that $\lambda_i(M_n)$ for GUE is asymptotically normally distributed around $\sqrt{n} \gamma_i$ with variance $\frac{2 \log n}{(4-\gamma_i^2) n}$, which is consistent with the above concentration results.  Using the four moment theorem, the result of Gustavsson was extended to other Wigner matrices in \cite{TVbulk}.  In particular, this gives a bound
$$ {\bf M} \lambda_i(M_n) = \sqrt{n} \gamma_i + o( \frac{\sqrt{\log n}}{\sqrt{n}} )$$
for the median uniformly in the bulk region $\delta n \leq i \leq (1-\delta) n$ for fixed $\delta > 0$, whenever the atom distribution of $M_n$ has vanishing third moment.  As a consequence of the recent results in \cite{EYY}, a similar result (with $n^\eps$ instead of $\sqrt{\log n}$) holds without the vanishing third moment hypothesis.
\end{remark}

Now we pass from the GUE case to more general Wigner matrices with vanishing third moment.  The main tool here is the three moment theorem (the second part of Theorem \ref{theorem:main2}).  We will also need a weak version of Lemma \ref{slick} in the non-GUE case:

\begin{lemma}[Weak concentration for Wigner]\label{slick-2} Let $M_n$ be a Wigner matrix, and let $I \subset \R$ be an interval.  Let $N_I$ be the counting function $N_I := \{ 1 \leq i \leq n: \frac{1}{\sqrt{n}} \lambda_i(M_n) \in I \}$.
Then for any fixed $\eps > 0$, one has
$$ \P( | N_I - n \int_I \rho_{sc}(x)\ dx| \geq n^{1/2+\eps} ) \leq n^{-100} $$
(say) uniformly in $I$, if $n$ is sufficiently large depending on $\eps$.
\end{lemma}

\begin{proof}  This follows from the concentration of measure approach first developed in \cite{GZ}, and then modified in \cite[Appendix F]{TVhard} to deal with exponentially decaying entries and with the discontinuous nature of the indicator function $1_I$.  As mentioned earlier, in the bulk case $I \subset [-2+\delta,2-\delta]$, stronger results of this type have also recently been obtained in \cite[Theorem 6.3]{EYY}.
\end{proof}

As a consequence of this lemma and our hypothesis $\min(i,n-i) > n^{1/2+\eps}$, one has the bound \eqref{lamn} with probability $1-O(n^{-100})$.  In other words, we can find a quantity $R$ comparable to $n^{1/2+\eps} \min(i,n+1-i)^{-1/3} n^{-1/6}$ such that
\begin{equation}\label{limi-2}
 \P( | \lambda_i(M_n) - \sqrt{n} \gamma_i | \geq R ) \leq n^{-100}.
\end{equation}
We now introduce the function
$$ G( x ) := \psi( \frac{x-n\gamma_i}{\sqrt{n} R} )$$
where $\psi: \R \to [0,1]$ is a smooth cutoff function supported on $[-2,2]$ that equals $\psi(x) := \frac{1}{10} x^2$ on $[-1,1]$.  Note that $\sqrt{n} R \geq \sqrt{n}$.  As such, for sufficiently small $c_0>0$ (independent of $\eps$), one easily verifies that
$$ |\nabla^j G(x)| \leq n^{-Cjc_0}$$
for $0 \leq j \leq 5$ and all $x \in \R$, where $C$ is the constant in the three moment theorem.  We may thus apply that theorem and conclude that
$$ \E G( \sqrt{n} \lambda_i(M_n) ) = \E G( \sqrt{n} \lambda_i(M'_n) ) + O(n^{-c_0})$$
for some absolute constant $c_0>0$, where $M'_n$ is drawn from GUE.  On the other hand, from \eqref{limi} one easily sees that
$$ \E G( \sqrt{n} \lambda_i(M'_n) ) = \frac{1}{10 R^2} \E | \lambda_i(M'_n) - \sqrt{n} \gamma_i |^2 + O(n^{-10})$$
(say), and similarly from \eqref{limi-2} one has
$$ \E G( \sqrt{n} \lambda_i(M_n) ) = \frac{1}{10 R^2} \E | \lambda_i(M_n) - \sqrt{n} \gamma_i |^2 + O(n^{-10}).$$
We conclude that
$$ \E | \lambda_i(M_n) - \sqrt{n} \gamma_i |^2 = \E | \lambda_i(M'_n) - \sqrt{n} \gamma_i |^2 + O( n^{-c_0} R^2 ) + O(n^{-10} ).$$
Substituting in the definition of $R$, we obtain Theorem \ref{stronglocal}.

\begin{remark}  From \eqref{limi} and the three moment theorem one can also show that for any Wigner matrix $M_n$ whose atom distribution has third vanishing moment, and any $1 \leq i \leq n$, one has
\begin{equation}\label{lma}
 \lambda_i(M_n) = \sqrt{n} \gamma_i + O( n^\eps \min(i,n+1-i)^{-1/3} n^{-1/6} ) 
\end{equation}
with probability $1-O_\eps(n^{-c})$ for some absolute constant $c>0$; in particular one has
$$
 {\bf M} \lambda_i(M_n) = \sqrt{n} \gamma_i + O_\eps( n^\eps \min(i,n+1-i)^{-1/3} n^{-1/6} ).
$$
We omit the details,which are similar to the above calculations and also to the proof of \cite[Theorem 32]{TVbulk} (which is essentially the bulk case of \eqref{lma}).  Heuristically, this suggests that one can take $c=1-\eps$ in \eqref{lamin}, which would be consistent with the results in \cite{EYY}; however, the available bound $O(n^{-c})$ of the tail probability for \eqref{lma} is too weak to make this heuristic rigorous.  Unfortunately, even if one assumes more than three matching moments, the methods of proof in \cite{TVbulk}, \cite{TVedge} do not seem strong enough to establish this conjecture; the main technical obstacle arises from the need to truncate away the event that an eigenvalue gap such as $\lambda_{i+1}(M_n) - \lambda_i(M_n)$ is unexpectedly small, such events occur with a probability of size $O(n^{-c})$ but no better.
\end{remark}

\begin{remark}
The vanishing third moment was needed since we compare to GUE.  One can omit this assumption if one can extend Lemma \ref{slick} to Johansson matrices (i.e. Wigner matrices whose atom distribution is gauss divisible, see \cite{Joh}). In fact, one only needs  this lemma with some $\eps <1/2$.  A similar remark applies to the proof of Theorem \ref{rateofconvergence} below.  The techniques in the recent paper \cite{EYY} also implies a version of Theorem \ref{stronglocal} in which no condition on the third moment is required (but with a slightly different right-hand side).
\end{remark} 

\subsection{Proof of Theorem \ref{rateofconvergence}}\label{rate-sec}

We now prove Theorem \ref{rateofconvergence}.  The method of proof is only a slight variant of that used above.

Fix $x$, and let $\eps> 0$ be a small absolute constant to be chosen later.  From Lemma \ref{slick-2} one has
$$ 
i_x - R \leq N_{[-2,x]} \leq i_x + R$$
with probability $1-O(n^{-100})$, where $R = O(n^{1/2+\eps})$ and
$$ i_x := n \int_{-\infty}^x \rho_{sc}(t)\ dt.$$
Inside the interval $I_x := [i_x - R, i_x - R] \cap [1,n]$, we locate $m = O(n^{1/4+2\eps})$ integers $i_1,\ldots,i_m$ such that every integer in $I_x$ lies within $O(n^{1/4-\eps})$ of one of the $i_j$.  Then with probability $1-O(n^{-100})$, one has
$$ 
N_{[-2,x]} = i_x - R + \sum_{i_x-R < j \leq i_x+R} {\bf I}( \frac{1}{\sqrt{n}} \lambda_{i_j}(M_n) \leq x )$$
where ${\bf I}(E)$ is the indicator of an event $E$, and we adopt the conventions that $\lambda_i(M_n)=-\infty$ for $i<1$ and $\lambda_i(M_n)=+\infty$ for $i>n$.  Taking expectations, we conclude that
$$ 
F_n(x) = \frac{1}{n} (i_x - R) + \frac{1}{n} \sum_{i_x-R < j \leq i_x+R} {\bf P}( \frac{1}{\sqrt{n}} \lambda_{j}(M_n) \leq x ) + O( n^{-10} )$$
(say).  Let $M'_n$ be sampled using GUE.  Using the three moment theorem as in \cite[Corollary 21]{TVbulk}, we have
$$
{\bf P}( \frac{1}{n} \lambda_{j}(M_n) \leq x ) \leq {\bf P}( \frac{1}{\sqrt  n} \lambda_{j}(M'_n) \leq x +n^{-1+c'} ) + O(n^{-c})$$
uniformly in $j$, for some small absolute $c, c' >0$ independent of $\eps$.  Using this bound, we conclude that
$$ 
F_n(x) \leq \frac{1}{n} (i_x - R) + \frac{1}{n} \sum_{i_x-R < j \leq i_x+R} {\bf P}( \frac{1}{\sqrt  n} \lambda_{j}(M'_n) \leq x + n^{-1+c'} ) + O( \frac{R}{n} n^{-c} ) + O(n^{-10}).$$
The second error term is $O( n^{-1/2-\eps} )$ if $\eps$ is small enough depending on $c$.  On the other hand, using Lemma \ref{slick-2} for $M'_n$ instead of $M_n$, we conclude from a variant of the above arguments that
$$ F'_n(x+n^{-1+c'}) = \frac{1}{n} (i_x - R) + \frac{1}{n} \sum_{i_x-R < j \leq i_x+R} {\bf P}( \frac{1}{\sqrt  n} \lambda_{j}(M'_n) \leq x + n^{-1+c} ) + O( n^{-10} ) $$
where $F'_n$ is the counterpart of $F_n$ for $M'_n$ instead of $M_n$.  We conclude that
$$ F_n(x) \leq F'_n(x+n^{-1+c'}) + O( n^{-1/2-\eps} );$$
a similar argument also gives
$$ F_n(x) \geq F'_n(x-n^{-1+c'}) - O( n^{-1/2-\eps} ).$$
On the other hand, from Lemma \ref{slick} one easily sees that
$$ F'_n(x) = \int_{-\infty}^x \rho_{sc}(t)\ dt + O(n^{-1+\eps})$$
for all $x$, and the claim follows.

\begin{remark}\label{remark:median} Using the three moment theorem, one can control the \emph{median} ${\bf M} N_{[-2,x]}$ with much higher accuracy than the \emph{mean} ${\bf E} N_{[-2,x]}$.  Indeed, using \eqref{lma} it is not difficult to show that
$$ {\bf M} N_I = n\int_I \rho_{sc}(x)\ dx + O_\eps( n^{\eps} )$$
uniformly for all intervals $I$ and any $\eps > 0$, assuming vanishing third moment of the atom distribution; we leave the details to the interested reader.  In view of this, it is reasonable to conjecture that one can take $c$ arbitrarily close to $1$ in Theorem \ref{rateofconvergence}.  In \cite{EYY}, this claim is established in the bulk region $I \subset [-2+\delta,2-\delta]$.
\end{remark}

\section{Proof of Theorem \ref{theorem:main}}\label{mainsec}

We now begin the proof of Theorem \ref{theorem:main}.  Let $M_n, M'_n, \lambda_i, \lambda'_i$ be as in that theorem.  The starting point is the following fourth moment calculation:

\begin{lemma}[Fourth moment calculation]\label{fourth}  
Set $\kappa_0 := \E (\eta^4) - \E((\eta')^4) $, thus $\kappa_0 \neq 0$ by hypothesis.  Then
$$ \sum_{i=1}^n \E(\lambda_i^4) - \E((\lambda'_i)^4) = 2 \kappa_0 (n^2-n).$$
\end{lemma}

\begin{proof}  We expand
\begin{align*}
\sum_{i=1}^n \E(\lambda_i^4)&= \E(\tr M_n^4)\\
&= \sum_{1 \leq a,b,c,d \leq n} \E \zeta_{ab} \zeta_{bc} \zeta_{cd} \zeta_{da}.
\end{align*}

Of course, there is a similar formula for $\sum_{i=1}^n \E \lambda_i'^4$, in which the $\zeta_{ij}$ are replaced by $\zeta'_{ij}$.

Consider the four sets $\{a,b\}, \{b,c\}, \{c,d\}, \{d,a\}$.  If one of these sets occurs with multiplicity one, then the expectation $\E \zeta_{ab} \zeta_{bc} \zeta_{cd} \zeta_{da}$ vanishes 
from the mean zero and independence properties of the coefficients of $M_n$.  If instead one has two pairs of sets occuring with multiplicity two, then the expectation of $\E \zeta_{ab} \zeta_{bc} \zeta_{cd} \zeta_{da}$ is equal to that of $\E \zeta'_{ab} \zeta'_{bc} \zeta'_{cd} \zeta'_{da}$.  From this we see that
$$
\sum_{i=1}^n \E(\lambda_i^4) - \sum_{i=1}^n \E((\lambda_i')^4)
= 2 \sum_{1 \leq a < b \leq n} \E(|\zeta_{ab}|^4) - \E(|\zeta'_{ab}|^4).$$
But a short calculation (using the fact that $\eta, \eta'$ match to third order) reveals that
$$ \E(|\zeta_{ab}|^4) - \E(|\zeta'_{ab}|^4) = 2 (\E(\eta^4) - \E((\eta')^4)).$$
The claim follows.
\end{proof}

The value $\max \{ |\lambda_1|, |\lambda_n | \}$ is called the spectral norm of $M_n$ and will be denoted by
$\|M_n \|$.  The following result is well-known:

\begin{lemma}[Concentration of the spectral norm]\label{V}  For any $A \geq 0$, one has
$$\P( \| M_n \| \geq 3 n^{1/2} ) = O_A(n^{-A}).$$
In particular,
$$\P( |\lambda_i(M_n)| \geq 3 n^{1/2} ) = O_A(n^{-A}) $$
and
$$ \E |\lambda_i(M_n)|^A = O_A( n^{A/2} ).$$
\end{lemma} 

\begin{proof} This follows easily from \eqref{vu} or \eqref{mox}, combined with \eqref{gammai-form} (as well as using crude estimates, such as H\"older's inequality, to deal with the rare tail event in which the estimates \eqref{vu} or \eqref{mox} fail).  Note that this argument allows us to replace the coefficient $3$ in the above large deviation inequality by $2+o(1)$, but we will not need this improvement here.  For even sharper concentration results, see the recent paper \cite{EYY}.
\end{proof}

We can now invoke Theorem \ref{stronglocal} to establish

\begin{proposition}[Fourth moment concentration]\label{fourthconc} We have
$$ |\sum_{i=1}^n \E(\lambda_i^4) - (\sqrt{n} \gamma_i)^4 - 4 (\sqrt{n} \gamma_i)^3 (\E \lambda_i-\sqrt{n}\gamma_i)| = O(n^{2-c})$$
for some absolute constant $c>0$, and similarly for $\lambda'_i$.
\end{proposition}

\begin{proof}  We begin with the Taylor expansion
$$\lambda_i^4 = (\sqrt{n} \gamma_i)^4 + 4 (\sqrt{n} \gamma_i)^3 (\lambda_i- \E \lambda_i) + O( |\lambda_i - \sqrt{n} \gamma_i|^2 (|\lambda_i| + \sqrt{n} \gamma_i)^2 ).$$
From Lemma \ref{V}, we see that with probability $1-O(n^{-100})$ (say), we have $|\lambda_i| + \sqrt{n} \gamma_i = O(\sqrt{n})$.  Taking expectations and summing using Theorem \ref{stronglocal} (and using crude estimates, such as H\"older's inequality, to deal with the tail event of probability $O(n^{-100})$) we obtain the claim.
\end{proof}

Combining Proposition \ref{fourthconc} with Lemma \ref{fourth} and the triangle inequality, we conclude (for $n$ large enough) that
$$ |\sum_{i=1}^n 4 \gamma_i^3 (\E \lambda_i - \E \lambda'_i)| \geq |\kappa_0| n^{1/2}.$$
Since $\gamma_i = O(1)$, Theorem \ref{theorem:main} now follows from the triangle inequality.

\begin{remark}  One can use \cite[Theorem 7.1]{EYY} as a substitute for Theorem \ref{stronglocal} in the arguments above.
\end{remark}

\section{Higher moment computations}\label{higher-sec}

In this section we discuss a higher moment computation that lead to Conjecture \ref{conj-asym}.  We will restate this conjecture at the end of this section for the reader's convenience.


\begin{lemma}[Higher moment computations]\label{higher}  $M_n, M'_n, \lambda_i, \lambda'_i$ be as in Theorem \ref{theorem:main}.
Set $\kappa_0 := \E (\eta^4) - \E((\eta')^4)$, and let $k \geq 0$ be an integer.  Then we have
$$ \sum_{i=1}^n \E(\lambda_i^k) - \E((\lambda'_i)^k) = ( 2 D_{(k-2)/2} \kappa_0 + O_k(n^{-1}) ) n^{k/2}$$
where the modified Catalan number $D_m$ is defined to be equal to
\begin{equation}\label{dm}
D_{m} = \binom{2m+2}{m-1} = \frac{(2m+2)!}{(m-1)! (m+3)!}
\end{equation}
when $m=1,2,\ldots$ is a positive integer, and $D_m=0$ otherwise, thus
$$ D_0 = 0; \quad D_1 = 1; \quad D_2 = 6; \quad D_3 = 28; \quad D_4 = 120; \quad \ldots$$
(OEIS sequence A002694\cite{oeis}).  
\end{lemma}

\begin{proof}  This is a standard moment method computation (which was the method used by Wigner to prove the semi-cirlce law \cite{BS};
this is also  related to the \emph{genus expansion} from string theory).   We have
$$ \sum_{i=1}^n \E(\lambda_i)^k = \sum_{1 \leq a_1,\ldots,a_k \leq n} \E \zeta_{a_1 a_2} \ldots \zeta_{a_k a_1}$$
and similarly for $\sum_{i=1}^n \E(\lambda'_i)^k$.

Consider the $k$ sets $\{ a_1,a_2\}, \{a_2,a_3\}, \ldots, \{a_k,a_1\}$.  If one of these sets appears with multiplicity one, then the expectation vanishes.  If none of the sets appears with multiplicity at least four, then the contributions to $\sum_{i=1}^n \E(\lambda_i)^k - \sum_{i=1}^n \E(\lambda'_i)^k$ cancel each other out.  Thus the only terms that survive are those in which each set appears with multiplicity at least two, and at least one set appears with multiplicity four.  In particular, there are at most $(k-2)/2$ distinct sets $\{a_i,a_{i+1}\}$ (with the convention $a_{k+1}=a_1$), and thus at most $k/2$ distinct values of $a_i$.  

If there are fewer than $k/2$ distinct values of $a_i$, then the total contribution here is easily seen to be $O_k(n^{k/2 - 1})$, which is acceptable.  Thus we may restrict attention to the case when there are exactly $k/2$ distinct values of $a_i$, which forces there to be exactly $(k-2)/2$ distinct sets $\{a_i,a_{i+1}\}$, and furthermore the (connected) graph formed by these edges cannot contain any cycles and is thus a tree.  Finally, each set $\{a_i,a_{i+1}\}$ must appear with multiplicity two, with the exception of one set that appears with multiplicity four.  The summand $\E \zeta_{a_1 a_2} \ldots \zeta_{a_k a_1}$ is then equal to $2\E \eta^4$, and similarly $\E \zeta'_{a_1 a_2} \ldots \zeta'_{a_k a_1}$ is then equal to $2\E (\eta')^4$.

We now assign each $a_i$ a label $j = j(a_i)$ from $1$ to $k/2$ by order of appearance; thus $a_1$ will be assigned a label $j(a_1)$ of $1$, the first $a_i$ that is distinct from $a_1$ will be assigned a label of $j(a_i) = 2$, the first $a_i$ has not already been labeled $1$ or $2$ will be labeled $3$, and so forth.  The closed path $\gamma = (( j(a_1), j(a_2) ), \ldots, (j(a_{k/2}),j(a_1)))$ then traverses a tree $T_\gamma$ of $(k-2)/2$ edges spanning the vertices $\{1,\ldots,k/2\}$, where the path $\gamma$ traverses each edge of $T_\gamma$ with multiplicity two, with the exception of one edge of $T_\gamma$ that is traversed four times.  Furthermore, the path $\gamma$ only encounters a vertex $j$ in $\{1,\ldots,k/2\}$ after it has first encountered $1,\ldots,j-1$; in particular, the starting (and ending) vertex of $\gamma$ is necessarily $1$.

Call a closed path $\gamma = ((j_1,j_2),\ldots,(j_{k/2},j_1))$ of length $k/2$ in $\{1,\ldots,k/2\}$ \emph{$4$-admissible} if it traverses a tree of $(k-2)/2$ edges spanning $\{1,\ldots,k/2\}$, so that each edge is traversed twice with the exception of one edge that is traversed four times, and such that each vertex $j$ is encountered only after encountering $1,\ldots,j-1$.  It is not hard to see that each such $4$-admissible path contributes 
$$( n^{k/2} + O_k( n^{k/2 - 1}) ) \times 2 \E \eta^4$$ 
to $\sum_{i=1}^n \E(\lambda_i)^k$, and similarly contributes
$$( n^{k/2} + O_k( n^{k/2 - 1}) ) \times 2 \E (\eta')^4$$ 
to $\sum_{i=1}^n \E(\lambda'_i)^k$.  Subtracting, we see that each $\gamma$ contributes
$$( n^{k/2} + O_k( n^{k/2 - 1}) ) \times 2 \kappa_0$$
to  $\sum_{i=1}^n \E(\lambda_i)^k - \sum_{i=1}^n \E(\lambda'_i)^k$.  Thus, it will suffice to show that for any $m$, the number of $4$-admissible paths on trees of $m$ edges is equal to $D_m$.

The claim is trivial unless $m$ is a positive integer.  We observe the recurrence
\begin{equation}\label{doom}
D_m = 2 \sum_{i,j \geq 0: i+j=m-1}^{m-1} C_i D_j + \sum_{i,j,k,l \geq 0: i+j+k+l=m-1} C_i C_j C_k C_l
\end{equation}
for $m=1,2,\ldots$, where 
$$ C_m := \frac{(2m)!}{m! (m+1)!}$$
are the Catalan numbers, thus
$$ C_0 = 0; \quad C_1 = 1; \quad C_2 = 2; \quad C_3 = 5; \quad C_4 = 14; \quad \ldots.$$
Indeed, writing
$$ c(x) := \sum_{m=0}^\infty C_m x^m = \frac{1 - \sqrt{1-4x}}{2x}$$
and
\begin{equation}\label{donkey}
d(x) := \sum_{m=0}^\infty D_m x^m = \frac{(1-\sqrt{1-4x})^4}{16 x^3 \sqrt{1-4x}}
\end{equation}
a brief calculation shows that
$$ d(x) = 2xc(x)d(x) + x c(x)^4$$
whence the claim.  

Call a path $((j_1,j_2),\ldots,(j_{2m},j_1))$ \emph{$2$-admissible} if it traverses a tree of $m$ edges spanning $\{1,\ldots,m+1\}$, such that each edge is traversed exactly twice, and each vertex $j$ is encountered only after encountering $1,\ldots,j-1$.  It is a classical fact that the number of $2$-admissible paths is $C_m$.

Now suppose inductively that there are $D_j$ $4$-admissible paths on trees of $j$ edges for all $j<m$.  It suffices to show that the number of $4$-admissible paths on trees of $m$ edges is given by the right-hand side of \eqref{doom}.
To do this, consider the first edge $(j_1,j_2)$ of an admissible path $\gamma$ on a tree with $m$ edges.  This edge is traversed either two or four times.  Suppose first that it is traversed two times.  Then one can split $\gamma$ into the following pieces: the first edge $(j_1,j_2)$, a (relabeled) $2$-admissible or $4$-admissible path on a tree with $i$ edges that starts and ends at $j_2$, a return edge $(j_2,j_1)$, and a (relabeled) $4$-admissible or $2$-admissible path on a tree with $j$ edges that is disjoint from the first tree that starts and ends at $j_1$, where $i,j \geq 0$ add up to $m-1$.  This case gives a net contribution of
$$ \sum_{i,j \geq 0: i+j=m-1} C_i D_j + D_i C_j = 2 \sum_{i,j \geq 0: i+j=m-1} C_i D_j$$
which is the first term of \eqref{doom}.

Now suppose that $(j_1,j_2)$ is traversed four times.  Then we can split $\gamma$ into the following pieces: the first edge $(j_1,j_2)$, a (relabeled) $2$-admissible path on a tree with $i$ edges that starts and ends at $j_2$, a return edge $(j_2,j_1)$, a (relabeled) $2$-admissible path with $j$ edges that starts and ends at $j_1$, a repeated edge $(j_1,j_2)$, a (relabeled) $2$-admissible path with $k$ edges that starts and ends at $j_2$, a repeated return edge $(j_2,j_1)$, and a (relabeled) $2$-admissible path on a tree with $l$ edges that starts and ends at $j_1$, where all trees are disjoint (except at their roots) and $i,j,k,l \geq 0$ add up to $m-1$.  This gives the second contribution to the right-hand side of \eqref{doom}.
\end{proof}

We now repeat the arguments from the previous section.  A routine generalisation of Proposition \ref{fourthconc} yields the bound
$$ |\sum_{i=1}^n \E(\lambda_i^k) - (\sqrt{n} \gamma_i)^k - k (\sqrt{n} \gamma_i)^{k-1} (\E \lambda_i-\sqrt{n} \gamma_i)| = O_k(n^{k-c})$$
for any $k \geq 1$ and some absolute constant $c>0$ (note that the left-hand side vanishes for $k=1$).  We conclude from this and Lemma \ref{higher} that
\begin{equation}\label{sli}
\frac{1}{n} \sum_{i=1}^n k \gamma_i^{k-1} \sqrt{n}(\E \lambda_i - \E \lambda'_i) = 2 D_{(k-2)/2} \kappa_0 + O_k(n^{-c}).
\end{equation}

Next, we observe from \eqref{donkey} that
$$ - \sum_{k=0}^\infty \frac{D_{(k-2)/2}}{z^k} = - \frac{(z-\sqrt{z^2-4})^4}{16 \sqrt{z^2-4}}$$
for $z$ near infinity, where we pick the branch of the square root of $\sqrt{z^2-4}$ that equals $z$ near infinity and is analytic away from the interval $[-2,2]$.  Then the right-hand side continues analytically to the exterior of this interval.  Calling this analytic function $f(z)$, we compute the jump formula
$$ \lim_{b \to 0} \frac{f(x+ib) - f(x-ib)}{2\pi i} = g(x)$$
where $g: \R \to \R$ vanishes outside of the interval $[-2,2]$, and is equal to
$$ g(x) := \frac{1}{2\pi} \frac{x^4-4x^2+2}{\sqrt{4-x^2}}$$
on this interval.  From the Cauchy integral formula we conclude the moment formula
$$ D_{(k-2)/2} = \int_{-2}^2 g(x) x^k\ dx$$
for $k=0,1,2,\ldots$.  The antiderivative of $g(x)$ is
$$ -\frac{1}{8\pi} (x^3-2x) \sqrt{4-x^2} = -\frac{1}{8} (x^3-2x) \rho_{sc}(x)$$
so by an integration by parts we have
$$ D_{(k-2)/2} = \int_{-2}^2 \frac{1}{8} (x^3-2x) k x^{k-1}\ \rho_{sc}(x) dx.$$
By Riemann integration (or more precisely, the trapezoid rule), the right-hand side is equal to
$$ \frac{1}{n} \sum_{i=1}^n \frac{1}{8} (\gamma_i^3 - 2 \gamma_i) k \gamma_i^{k-1} + O_k(n^{-c})$$
for some absolute constant $c>0$.

Thus if we introduce the normalised shift
$$ s_i := \sqrt{n}(\E \lambda_i - \E \lambda'_i) - \frac{1}{4} (\gamma_i^3 - 2 \gamma_i) \kappa_0$$
we can rewrite \eqref{sli} as
\begin{equation}\label{sli-2}
\frac{1}{n} \sum_{i=1}^n k \gamma_i^{k-1} s_i = O_k(n^{-c}).
\end{equation}
This suggests (but does not rigorously prove\footnote{Specifically, the difficulty is that there could be cancellation between nearby values of $s_i$.  If one could show some assertion to the effect that $s_i \approx s_{i'}$ when $i, i'$ are close together, then this would go a long way towards establishing Conjecture \ref{conj-asym} below.}) that the $s_i$ are small, of size $O(n^{-c})$, at least in the bulk region $\delta n \leq i \leq (1-\delta) n$.  In particular, we are led to the conjecture (Conjecture \ref{conj-asym})  that
$$ (\E \lambda_i - \E \lambda'_i) = \frac{1}{4 \sqrt{n}} (\gamma_i^3 - 2 \gamma_i) \kappa_0 + O(n^{-1/2-c})$$
and this in turn suggests the following asymptotic for the expected value of $\lambda_i$:

\begin{conjecture}[Conjectured asymptotic]\label{conj-asym}  For $\delta n \leq i \leq (1-\delta) n$ with $\delta>0$ fixed, one has
$$ \E \lambda_i = n^{1/2} \gamma_i + n^{-1/2} C_{i,n} + \frac{1}{4 \sqrt{n}} (\gamma_i^3 - 2 \gamma_i) \E \eta^4 + O_\delta(n^{-1/2-c} )$$
for some absolute constant $c>0$, where $C_{i,n}$ is some bounded quantity depending only on $i, n$ (and is in particular independent of $\eta$).  The same statement should also be true for the median ${\bf M} \lambda_i$.
\end{conjecture}

The bound on $C_{i,n}$ is plausible in view of results such as Lemma \ref{slick}.  It should in fact be possible to obtain (at least conjecturally) a precise value for $C_{i,n}$ from an analysis of the GUE case.  Such an asymptotic would demonstrate more precisely the dependence of the $i^{th}$ eigenvalue on the fourth moment $\E \eta^4$ at the scale $\Theta(n^{-1/2})$ of the mean eigenvalue spacing.

\end{document}